\documentclass[12pt]{amsart}
\pdfoutput=1
\usepackage[english]{babel}
\usepackage[top=1.6in, bottom=1.6in, left=1.5in, right=1.5in]{geometry}
\usepackage{amsthm,amsmath,amssymb,amsfonts}
\usepackage[colorlinks=true,urlcolor=blue,linkcolor=blue,citecolor=blue]{hyperref}
\usepackage[alphabetic]{amsrefs}
\usepackage{mathtools} 
\usepackage{graphicx}
\usepackage{setspace}

\newtheorem{theorem}{Theorem}[section]
\newtheorem{lemma}[theorem]{Lemma}
\newtheorem{proposition}[theorem]{Proposition}
\newtheorem{corollary}[theorem]{Corollary}
\theoremstyle{remark}

\theoremstyle{definition}
\newtheorem{definition}[theorem]{Definition}

\usepackage{xcolor}
\newcommand{\Addresses}{{
  \bigskip
  \footnotesize
  \textsc{
Department of Mathematics \& Statistics, McMaster University\\
Hamilton, Ontario,
Canada L8S 4K1
  }\par\nopagebreak
      \textit{E-mail address}: \texttt{chenj293@mcmaster.ca}
  }}
\title{Connected sum and crossing numbers of flat virtual knots}
\author{Jie Chen}
\date{}
\begin{document}
\begin{abstract}
         The connected sum of two flat virtual knots depends on the choice of diagrams and basepoints. We show that any minimal crossing diagram of a composite flat virtual knot is a connected sum diagram. We also show the crossing number of flat virtual knots is super-additive under connected sum.
\end{abstract}

\maketitle
\section{Introduction}
A famous conjecture in classical knot theory asserts that the crossing number is additive under connected sum. It is known that the crossing number satisfies $cr(K_1\# K_2)\le cr(K_1)+cr(K_2)$, so it is sub-additive, and additivity has been established for certain families of knots (e.g.~alternating knots). In general, though, it remains a difficult open problem to prove additivity for all knots. 

In this paper, we consider the analogous question for \emph{flat virtual knots}. Flat virtual knots were first introduced by Kauffman \cite{kauffman99}, and they were also studied by Turaev  \cite{turaev04}, who called them \emph{virtual strings}. A flat virtual knot (or virtual string) can be represented as a regular immersed curve $\alpha: S^1 \to \Sigma$ in a compact oriented surface $\Sigma$, up to homotopy and stabilization \cite{Carter-Kamada-Saito}. The operation of connected sum is not well-defined; it depends on the choice of oriented diagrams and basepoints on the flat virtual knots. 

The main result in this paper is Theorem~\ref{composite2}; it shows that any minimal crossing diagram of a composite flat virtual knot is a connected sum diagram. By way of application, we use it to show that the crossing number is super-additive under connected sum of flat virtual knots. Namely, we show that $cr(\alpha_1\# \alpha_2)\ge cr(\alpha_1)+cr(\alpha_2)$ (Corollary \ref{cor:superadd}). This inequality holds for any composite flat virtual knot obtained as the connected sum of $\alpha_1$ and $\alpha_2$, regardless of the choice of diagrams and basepoints. There are examples to show the inequality is sometimes strict, but if the connected sum is formed using two minimal crossing diagrams, then Proposition~\ref{prop:min} applies to give an equality of crossing numbers.

Here is a brief overview of the rest of this paper. Section~\ref{sec-prelim} contains the basic definitions for flat virtual knots, flat Reidemeister moves, and Gauss diagrams. It also presents the monotonicity result (Theorem~\ref{thm:monotonicity}) for flat virtual knots. Section~\ref{section-crossing-number} reviews the construction of the connected sum for flat virtual knots. It also introduces long flat virtual knots and
concludes with statements and proofs of the main results (Theorem~\ref{composite2} and Corollary~\ref{cor:superadd}). 

\section{Preliminaries} \label{sec-prelim}
In this section, we introduce flat virtual knots and links. They can be described alternatively in terms of flat virtual link diagrams or Gauss diagrams.
 The flat Reidemeister moves 
generate an equivalence relation called homotopy. 
\begin{definition}
A \emph{flat virtual link diagram} is a $4$-valent embedded planar graph, 
where each vertex is either a \emph{flat crossing} or virtual crossing
as in Figure~\ref{fig:fcr}.
\end{definition}
\begin{figure}[ht]
    \centering
    \includegraphics[scale=1.4]{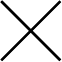} \qquad
    \includegraphics[scale=1.4]{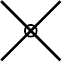}
    \caption[Flat and virtual crossings]{A flat crossing (left) and virtual crossing (right) }
    \label{fig:fcr}
\end{figure}

\begin{definition}
    Two flat virtual link diagrams are said to be \emph{homotopic} if they are related by a finite sequence of 
    \emph{flat Reidemeister moves} shown in Figure~\ref{fig:frmove}
    along with planar isotopies.
    \label{def:fr}
\end{definition}

For a flat virtual link diagram $D$, its \emph{crossing number} is denoted $cr(D)$ and defined as the number of flat crossings in $D$. Note that the virtual crossings of $D$ do not contribute to $cr(D).$ The crossing number of a flat virtual link $\alpha$ is denoted $cr(\alpha)$ and defined as the minimum, over all diagrams $D$ for $\alpha$, of the crossing number of $D$. Thus, $cr(\alpha)$ is an invariant of the flat virtual link. 

Notice that, for the various moves FR1--FR4 and VR1--VR3 in Figure~\ref{fig:frmove}; only two of them (FR1, FR2) alter the crossing number of the diagram. An FR1 or FR2 move is said to be \emph{increasing} if it increases the crossing number of the diagram, and it is said to be \emph{decreasing} if it decreases the crossing number of the diagram.

Flat virtual links arise naturally from virtual links by flattening all of the classical crossings \cite{kauffman99}. When we refer to a component of a flat virtual link, we mean a connected component of a virtual link overlying it. In this paper, we will be mainly interested in oriented flat virtual knots, which are oriented flat virtual links with one component.

\begin{figure}[hb]
    \centering
    \includegraphics[scale=1.00]{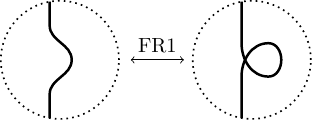}
    \hspace{20pt}
    \includegraphics[scale=1.00]{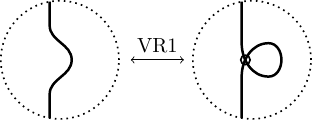}\\
    \vspace{12pt}
    \includegraphics[scale=1.00]{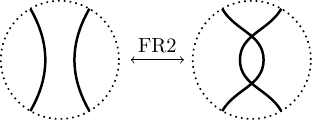}
    \hspace{20pt}
    \includegraphics[scale=1.00]{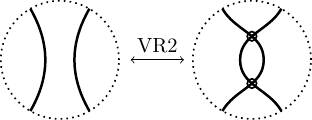}\\
    \vspace{12pt}
    \includegraphics[scale=1.00]{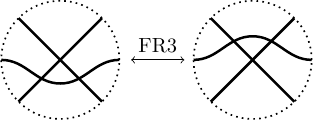}
    \hspace{20pt}
    \includegraphics[scale=1.00]{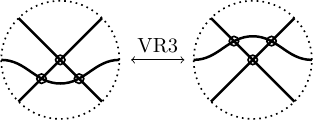}
    \vspace{12pt}
    
    \includegraphics[scale=0.95]{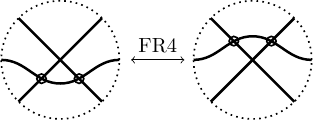}
    \caption[Flat Reidemeister moves]{Flat Reidemeister moves}
    \label{fig:frmove}
\end{figure}

Gauss initiated the study of planar curves and he developed a notation to encode them using double occurrence words which are now known as \emph{Gauss words}. A Gauss diagram is a graphical representation of a double occurrence word, and (signed) Gauss diagrams provide a useful  notation for immersed curves in surfaces, virtual strings, and virtual knots
\cites{Carter, turaev04, kauffman99}.

\begin{definition}
For an oriented flat virtual knot diagram with $n$ classical crossings, its \emph{Gauss diagram}
is a counterclockwise oriented circle with $2n$ points partitioned into $n$ ordered pairs. The pair of points is indicated by drawing an arrow from the first point to the second.
\end{definition}
The circle represents the knot, and each pair of points (or arrow) represents a crossing. The direction of the arrow is determined by orientations as shown in Figure~\ref{fig:fgauss2}. Specifically, the arrow points from the first arc to the second if and only if the tangents to the arcs form an oriented basis for the plane.
The positions of the arrows determine the order in which the flat crossings are traversed, 
and the virtual crossings are a by-product of drawing a non-planar graph in the plane. Their placement is arbitrary, which is a consequence of VR1, VR2, VR3, and FR4, and the Gauss diagram provides no information about the virtual crossings.

Figure~\ref{fig:gauss5} shows two flat virtual knots with their Gauss diagrams. On the left is the flat virtual trefoil, which is flat trivial. On the right is the flat virtual knot 3.1, which is nontrivial and the first nontrivial flat virtual knot in the tabulation \cite{flatknotinfo}.

\begin{figure}[ht]
    \centering
    \includegraphics[scale=1.00]{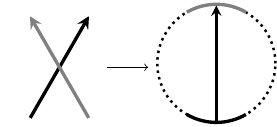}
   \caption{Flat crossings to Gauss diagram arrows}
    \label{fig:fgauss2}
\end{figure}

\begin{figure}[ht]
    \centering
    \includegraphics[scale=0.90]{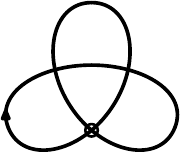}
    \quad\includegraphics[scale=0.70]{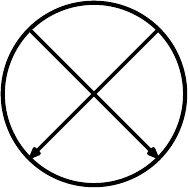}
    \qquad \qquad
    \includegraphics[scale=1.35]{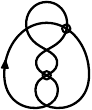}
    \quad\includegraphics[scale=0.70]{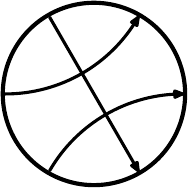}
   
    \caption{Two flat knots with Gauss diagrams}
    \label{fig:gauss5}
\end{figure}
Every oriented flat virtual knot diagrams determines an associated Gauss diagram, and conversely any Gauss diagram determines an oriented flat virtual knot diagram. The flat Reidemeister moves of Figure~\ref{fig:frmove} can be translated into moves on Gauss diagrams. Since the Gauss diagram records arcs with orientation, the translated moves FR1, FR2, and FR3 have several possible variants. Figure~\ref{fig:gfrmove12} shows the variant moves for FR1 and FR2, and Figure~\ref{fig:gfrmove3} shows the variant moves for FR3.
Note that VR1, VR2, VR3 and FR4 do not change the Gauss diagram. 
We continue to refer to variant FR1 and FR2 moves as being \emph{increasing} (or \emph{decreasing}) according to whether one or two arrows appear (or disappear) under its application.

For any classical knot, there is a sequence of Reidemeister moves to convert any given diagram into a minimal crossing diagram. However, the sequence may increase the number of crossings before achieving a minimal crossing diagram; see \cite{Henrich-Kauffman} for examples of unknot diagrams exhibiting this phenomenon.

For flat virtual knots, the following result shows that one can convert any flat virtual knot diagram into a minimal crossing diagram without introducing new crossings, i.e., without using increasing FR1 or FR2 moves. The following theorem is a statement of the monotonicity result for flat virtual knots, \cites{hassscott,manturovbook, cahn,david}.

\begin{theorem}
    For any flat virtual knot diagram, there exists a sequence of flat Reidemeister moves such that the number of crossings monotonically decreases until one achieves a minimal crossing diagram.
    Further, any two minimal crossing diagrams of the same flat virtual knot are related by a sequence of FR3 moves.
    \label{thm:monotonicity}
\end{theorem}

The monotonicity result for flat virtual knots and links has a somewhat storied history. In \cite{hassscott}, Hass and Scott showed that any curve on a surface can be monotonically reduced to one with minimal intersection number. In \cite{MR2008880}, Kadokami claimed that any flat virtual link diagram can be monotonically reduced to a minimal crossing diagram, but Gibson found counterexamples in \cite{gibson}.  In \cite{cahn}, Cahn proved Kadokami's claim for flat virtual knots, and in \cite{david}, Freund gave a complete solution by proving Kadokami's claim for non-parallel flat virtual links.

\begin{figure}[ht!]
    \centering
    \includegraphics[scale=0.95]{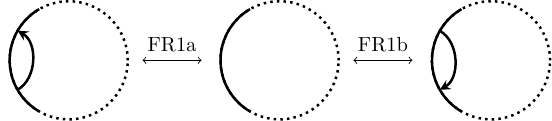}
    \\ \vspace{4pt}
    \includegraphics[scale=0.95]{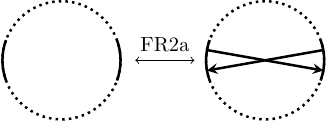}
    \\ \vspace{4pt}
    \includegraphics[scale=0.95]{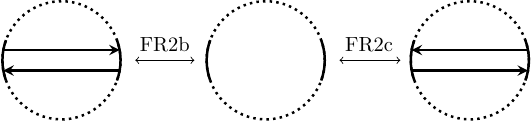}
    \caption{The variant FR1 and FR2 moves on Gauss diagrams}
    \label{fig:gfrmove12}
\end{figure}

\begin{figure}[ht!]
    \centering
    \includegraphics[scale=0.95]{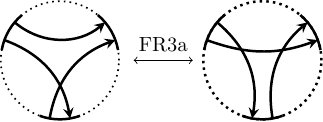}
    \hspace{15pt}
    \includegraphics[scale=0.95]{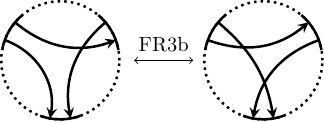}\\
    \vspace{6pt}
    \includegraphics[scale=0.95]{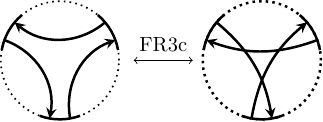}
    \hspace{15pt}
    \includegraphics[scale=0.95]{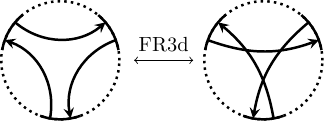}
   \caption{The variant FR3 moves on Gauss diagrams}
    \label{fig:gfrmove3}
\end{figure}

\section{Composite flat virtual knots and crossing number} \label{section-crossing-number}
In this section, we define the connected sum for oriented flat virtual knots and consider the crossing number of composite flat virtual knots. 
We will show that every minimal crossing diagram of a composite flat virtual knot is a connected sum diagram.

We use  $D^\bullet$ to denote an oriented flat virtual knot diagram $D$ together with a choice of basepoint $p$ on it.
Given two oriented diagrams $D_1, D_2$ of flat virtual knots with basepoints $p_1,p_2$, the connected sum is denoted $D_1^\bullet \# D_2^\bullet$. 
It is the diagram obtained by joining $D_1$ to $D_2$ at the basepoints so that the orientations on the joined arcs match up. When both of the basepoints $p_1$ and $p_2$ are adjacent to the unbounded regions of $D_1$ and $D_2$, respectively, then one connects the diagrams in the usual way. If either $p_1$ or $p_2$ is not adjacent to the unbounded region, then we use virtual crossings to connect $D_1$ to $D_2$ at the basepoints without introducing any new flat crossings.

The connected sum operation is easy to visualize for Gauss diagrams. It involves connecting the two circles in an orientation preserving way at the basepoints. For example, see Figure~\ref{fig:longflat2}. Recall that Gauss diagrams are always oriented counterclockwise and do not record virtual crossings. 

The operation of connected sum is not well-defined for flat virtual knots;  it depends on the diagrams, orientations, and choice of basepoints, see Figure~\ref{fig:longflat2}.

This example also shows that the decomposition theorem, a theorem stating every flat virtual knot can be uniquely decomposed as a connected sum of nontrivial prime flat virtual knots, does not exist for flat virtual knots; see \cite{Matveev}.

\begin{figure}[ht]
    \centering
\includegraphics[scale=0.85]{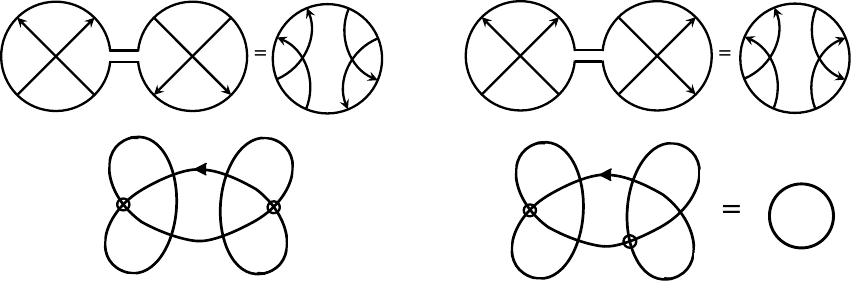}
\caption{The connected sum operation depends on basepoints}
    \label{fig:longflat2}
\end{figure}

\begin{definition}
    A flat virtual knot $\alpha$ is said to be \emph{composite} if it can be represented  by a connected sum diagram $D_1^\bullet \# D_2^\bullet$ such that $\alpha$ is not equivalent to $D_1$ or $D_2$.
A flat virtual knot $\alpha$ is said to be \emph{prime} if it is not a composite flat virtual knot.  
\end{definition}

\begin{definition}
    Let $D_1,D_2$ be two Gauss diagrams for flat virtual knots. The  \emph{set of the permutant diagrams} is defined to be the set of the connected sum Gauss diagrams
$$[D_1 \# D_2] = \{D_1^\bullet \# D_2^\bullet \mid  D_i^\bullet \text{ equals $D_i$ with a choice of basepoint}\}.$$
\label{def:perm}
\end{definition}

Note that the connected sum of two trivial flat virtual knot diagrams can produce a nontrivial flat virtual knot.
 For example,  the flat virtual knots in Figure \ref{fig:fk410} are nontrivial connected sums of two and three trivial diagrams, respectively. Similarly, one can construct a nontrivial flat virtual knot as the connected sum of $n$ trivial diagrams. These flat virtual knots can be shown to be different by computing their based matrices and associated characteristic polynomials, see  \cites{turaev04,chen-thesis}. 
\begin{figure}[ht]
    \centering
\includegraphics[scale=1.20]{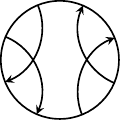}
\hspace{22pt}
\includegraphics[scale=1.20]{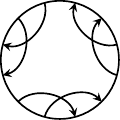}
\caption{Flat virtual knots 4.10 and 6.2064}
    \label{fig:fk410}
\end{figure}

\begin{proposition}
  Let $D_1$ and $D_2$ be two minimal crossing Gauss diagrams. Then any permutant diagram $D\in[D_1\#D_2]$ is a minimal crossing number diagram of its flat virtual knot type.
  \label{prop:min}
\end{proposition}

\begin{proof}
    For the sake of contradiction, suppose the crossing number of $D$ can be reduced.
   By Theorem~\ref{thm:monotonicity}, there is a finite sequence of flat Reidemeister moves that can include any of the variants of the FR3 moves in Figure~\ref{fig:gfrmove3} but only the decreasing variants of the FR1 and FR2 moves in Figure~\ref{fig:gfrmove12}.

    Suppose FR3  can be applied on three arrows, with one from $D_1$ and two from $D_2$. 
        This can only happen in the versions a, b, c and d of FR3, in the left-to-right direction, as shown in Figure~\ref{fig:gfrmove3}.
   However, such Gauss diagrams of $D_1$ or $D_2$ are not minimal crossing. For example, as shown in Figure~\ref{fig:gfrex}, if the connected sum is formed at the red dashed line, then applying an FR1 move on $D_1$ reduces the crossing number. The same argument applies to FR3b, FR3c, FR3d as well.
   Therefore, no FR3 moves can be applied with one arrow from $D_1$ and two arrows from $D_2$. The same argument with $D_1,D_2$ switched implies that if $D_1, D_2$ are minimal crossing Gauss diagrams, then any FR3 move must have all three arrows in $D_1$ or in $D_2.$ In particular, the new diagram obtained after FR3 moves will be a connected sum diagram of the form $D' \in [D_1'\#D_2]$ or $[D_1\#D'_2]$, where $D_1', D_2'$ are obtained by applying FR3 to $D_1,D_2,$ respectively. Further, notice that $D_1'$ and $D_2'$ are both still minimal crossing diagrams.
   
   Now consider a decreasing FR1 or FR2 move applied to a diagram $D \in [D_1 \# D_2]$.
  The assumption that $D_1$ and $D_2$  are minimal crossing diagrams rules out the possibility that in a decreasing RM2 move on $D_1\# D_2$, one arrow is from $D_1$ and the other from $D_2$.
   Then the arrows involved in the decreasing FR1 or FR2 move must belong to
   either $D_1$ or $D_2$.  However, this contradicts our assumption that $D_1$ and $D_2$ were minimal crossing diagrams. Therefore, no decreasing FR1 or FR2 moves can be applied to $D$ to reduce its crossing number.
\end{proof}

\begin{figure}[ht]
    \centering
    \includegraphics[scale=1.2]{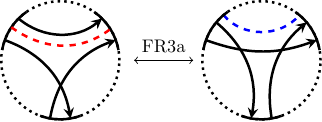}
   \caption{FR3a move on a Gauss diagram}
    \label{fig:gfrex}
\end{figure}

In order to state the next result, we recall the definition of long flat virtual knots.
A \emph{long flat virtual knot diagram} is an oriented $4$-valent planar graph with one component that coincides with the $x$-axis outside of some compact set, where it has flat and virtual crossings;
see Figure~\ref{fig:fk52long}.
Two long flat virtual knot diagrams are said to be \emph{equivalent} if they are related by a finite sequence of flat Reidemeister moves and planar isotopies.
    A \emph{long flat virtual knot} is an equivalence class of long flat virtual knot diagrams.

The \emph{closure} of a long flat virtual knot is the flat virtual knot obtained by replacing the $x$-axis outside the compact set with a half circle. 
If $D$ is a flat virtual knot diagram, then removing a point from $D$ gives a long flat virtual knot diagram $D^\bullet$, and the  closure of $D^\bullet$ is $D$.

\begin{figure}[ht]
    \centering
\includegraphics[scale=2.5]{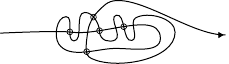}
    \caption{A long flat virtual knot diagram}
    \label{fig:fk52long}
\end{figure}

\begin{lemma}
Let $D$ be a connected sum diagram $D=D_1^\bullet \#D_2^\bullet$, 
where $D_1^\bullet$ and $D_2^\bullet$ are both nontrivial as long flat virtual knots.
    Among all flat Reidemeister moves, only increasing FR2 moves can change $D$ to a non-connected sum diagram.
    \label{lem:composite}
\end{lemma}
\begin{proof}
    Observe that if a diagram is a connected sum before applying a decreasing FR1 or FR2 move, then it continues to be a connected sum diagram after performing the FR1 or FR2 move, see Figure~\ref{fig:gfrmove12}.
    Further observe that such a diagram will remain a connected sum under an increasing FR1 move.
    So it suffices to show that no FR3 move can transform a connected sum diagram to a non-connected sum.
    This is obviously true if all three arrows involved in the FR3 move belong to one of $D_1$ or $D_2,$ so we consider only the case where one of the arrows is in $D_1$ and the other two are in $D_2.$
    As shown Figure~\ref{fig:gfrex}, if the connected sum is realized along the red dashed line for the diagram on the left, then the blue dashed line realizes the connected sum for the diagram on the right. This is the argument for the variant FR3a, and the argument for the other variants FR3 moves is similar.

\end{proof}

\begin{theorem}
  Any minimal crossing diagram  of a composite flat virtual knot  is a connected sum diagram.
    \label{composite2}
\end{theorem}
\begin{proof}
    A composite flat virtual knot has a Gauss diagram   $D=D_1^\bullet \#D_2^\bullet$, where $D_1^\bullet$ and $D_2^\bullet$ are two Gauss diagrams with a choice of basepoint on them, and removing the basepoint from either of them produces a nontrivial long flat virtual knot.
Assume $D'$ is a minimal crossing diagram that is flat equivalent to $D$.
    By Theorem~\ref{thm:monotonicity}, $D'$ is obtained from $D$ by finitely many FR3 moves and decreasing FR1 and FR2 moves. Then by Lemma~\ref{lem:composite}, $D'$ is a connected sum diagram.
\end{proof}

Note that Theorem~\ref{composite2} does not require that the two summands $D_1$ and $D_2$  be minimal crossing diagrams.

Primality is a flat virtual knot invariant, and Theorem~\ref{composite2} implies that a flat virtual knot is prime if and only if any minimal crossing representative is a prime diagram. Combined with the monotonicity result (Theorem~\ref{thm:monotonicity}), this leads to a simple algorithm to determine whether a given flat virtual knot is prime or composite. Since the representatives in the recent tabulation \cite{flatknotinfo} of flat virtual knots up to 8 crossings are all minimal crossing, one can check primeness of the flat virtual knots from \cite{flatknotinfo} by inspection. However, in classification, one of the most difficult problems is to distinguish composite flat virtual knots that are permutant (Definition~\ref{def:perm}). The reason is that most flat knot invariants are equivalent on permutant pairs. What is needed are new invariants that are capable of distinguishing inequivalent permutant pairs.

The next result is a direct consequence of Theorem~\ref{composite2}; and it shows that the crossing number is super-additive for connected sums of flat virtual knots.

\begin{corollary}
Let $D,D_1,D_2$ be Gauss diagrams of the flat virtual knots $\alpha,\alpha_1,\alpha_2$, respectively. If $D$ is a connected sum of $D_1$ and $D_2$,
    then the crossing number satisfies
    \begin{equation} \label{eqn-crossing_number} 
    cr(\alpha)\ge cr(\alpha_1)+cr(\alpha_2)
    \end{equation} 
    \label{cor:superadd}
\end{corollary}

\begin{proof}
    The inequality holds if  $cr(\alpha_i)=0$ for $i=1$ or $i=2$.

    Assume $\alpha$ is a composite knot and $cr(\alpha_i)>0$ for $i=1,2$. 
    Then by the argument of Theorem~\ref{composite2}, a minimal crossing diagram $D'$ of $\alpha$ is a connected sum  
    of $D'_1$ and $D'_2$, where $D'_1, D'_2$ are Gauss diagrams of $\alpha_1, \alpha_2$, respectively.
Then we have
$$cr(\alpha)=cr(D')= cr(D'_1)+cr(D'_2)\ge cr(\alpha_1)+cr(\alpha_2).$$
\end{proof}

There are examples showing that the inequality in Equation \eqref{eqn-crossing_number} is sometimes strict. For example, Figure~\ref{fig:longflat2} shows two connected sums of Gauss diagrams and their corresponding flat virtual knots. Both are connected sum diagrams of the flat virtual unknot, but the one on the right is trivial and the one on the left is a nontrivial flat virtual knot with four crossings.
Note that Proposition~\ref{prop:min} implies that   equality holds in Equation \eqref{eqn-crossing_number} when the summands $D_1$ and $D_2$ are minimal crossing diagrams.

\section*{Acknowledgements}
The main results in this paper are derived from the author's Ph.D. thesis \cite{chen-thesis}, written under the supervision of Hans Boden at McMaster University.

\Addresses
\end{document}